\documentclass[12pt]{amsart}
\usepackage{amssymb,amsmath,amscd,enumerate}
\usepackage{tikz}

\usepackage{capt-of}
\usepackage[colorlinks=true,linkcolor=blue,citecolor=blue]{hyperref}
\usepackage{graphicx}
\usepackage{hyperref}
\usepackage[english]{babel}

\usepackage{float} 
\usepackage{mathrsfs}
\usetikzlibrary{arrows}
 \numberwithin{equation}{section}
 \newtheorem{thm}{\bf Theorem}[section]
 \newtheorem{lem}[thm]{\bf Lemma}
 
 \newtheorem{cor}[thm]{\bf Corollary}
 \newtheorem{prop}[thm]{\bf Proposition}
 \newtheorem{prob}[thm]{\bf Problem}
 \theoremstyle{definition}
 \newtheorem{defn}[thm]{Definition}
 
 \newtheorem{exam}[thm]{Example}
 
 \newtheorem{rem}[thm]{Remark}
 \newtheorem{quest}[thm]{Question}
\newtheorem*{thm*}{Theorem}
\usepackage{color}

\DeclareMathOperator{\rad}{rad}
\DeclareMathOperator{\ecc}{ecc}

\DeclareMathOperator{\diam}{diam}

\DeclareMathOperator{\girth}{girth}

\usepackage[normalem]{ulem}

\begin{document}
\title[Well-covered powers of graph]{On graphs whose powers are well-covered}
\date{}

\author[My Hanh Pham]{My Hanh Pham}
\address{Faculty of Education, An Giang University\\Vietnam National University Ho Chi Minh City\\ Long Xuyen, An Giang, Vietnam}
\email{pmhanh@agu.edu.vn}

\author[Thanh Vu]{Thanh Vu}
\address{Institute of Mathematics, VAST, 18 Hoang Quoc Viet, Hanoi, Vietnam}
\email{vuqthanh@gmail.com}

\let\thefootnote\relax\footnotetext{*Corresponding Author: Thanh Vu}

\dedicatory{Dedicated to Professor Le Tuan Hoa on the occasion of his 70th birthday}
 
\begin{abstract} 
In this paper, we investigate graphs whose powers are all well-covered. We obtain complete classifications for several graph families, including chordal, cochordal, and very well-covered graphs, as well as graphs of girth at least $5$.
\end{abstract}

\maketitle

\section{Introduction}
  \label{sec1} 

Let $G$ be a simple graph. A subset $C \subseteq V(G)$ is called a \emph{vertex cover} of $G$ if $C$ intersects every edge of $G$, and a \emph{minimal vertex cover} if no proper subset of $C$ is a vertex cover. A graph $G$ is said to be \emph{well-covered} if all of its minimal vertex covers have the same cardinality.

Well-coveredness is a central topic in graph theory~\cite{BM, P}, with structural characterizations established for several prominent classes, including chordal graphs~\cite{PTV}, bipartite graphs~\cite{R}, and graphs with $\operatorname{girth}(G) \ge 5$~\cite{FHN}. Graph powers form another fundamental notion in structural graph theory: for a positive integer $d$, the $d$-th power $G^d$ has vertex set $V(G^d) = V(G)$, and two distinct vertices $u, v \in V(G)$ are adjacent in $G^d$ if and only if they are connected by a path of length at most $d$ in $G$. In this paper, we study the following classification problem:

\begin{prob}
Classify all simple graphs $G$ such that $G^d$ is well-covered for every $d \ge 1$.
\end{prob}

In the bipartite setting, Deniz~\cite{D} showed that if both $G$ and $G^2$ are well-covered, then every connected component of $G$ is isomorphic to either $K_1$ or $K_{n,n}$. This indicates that requiring $G^d$ to be well-covered for all $d \ge 1$ imposes strong structural restrictions on $G$. Since a graph $G$ is well-covered if and only if each of its connected components is well-covered, we may assume without loss of generality that $G$ is connected and non-trivial (that is, has no isolated vertices).

Our first main results classify such graphs within two important families: chordal graphs and very well-covered graphs.

\begin{thm}\label{thm_chordal}
Let $G$ be a connected chordal graph with no isolated vertices. Then the following conditions are equivalent:
\begin{enumerate}
    \item $G^d$ is well-covered for all $d \ge 1$;
    \item $G$ and $G^2$ are well-covered;
    \item $G$ is a complete graph $K_n$ for some $n \ge 2$.
\end{enumerate}
\end{thm}

\begin{thm}\label{thm_very_well}
Let $G$ be a connected, very well-covered graph with no isolated vertices. Then the following conditions are equivalent:
\begin{enumerate}
    \item $G^d$ is well-covered for all $d \ge 1$;
    \item $G$ and $G^2$ are well-covered;
    \item $G$ is isomorphic to $K_{n,n}$ for some $n \ge 1$.
\end{enumerate}
\end{thm}

Recall that for a connected graph $G$, the distance $d_G(u, v)$ between two vertices $u, v \in V(G)$ is the length of a shortest path connecting them, and the \emph{diameter} of $G$, denoted by $\operatorname{diam}(G)$, is
\[
\operatorname{diam}(G) = \max \{ d_G(u, v) \mid u, v \in V(G) \}.
\]
Notice that if $\operatorname{diam}(G) \le 2$, then $G^2$ (and indeed $G^d$ for all $d \ge 2$) is a complete graph, which is trivially well-covered. Our next main result shows that for cochordal graphs, this diameter condition completely characterizes when all powers remain well-covered.

\begin{thm}\label{thm_cochor}
Let $G$ be a connected cochordal graph with no isolated vertices. Then the following conditions are equivalent:
\begin{enumerate}
    \item $G^d$ is well-covered for all $d \ge 1$;
    \item $G$ and $G^2$ are well-covered;
    \item $G$ is well-covered and $\operatorname{diam}(G) \le 2$.
\end{enumerate}
\end{thm}

For graphs with girth at least $5$, we obtain the following classification.

\begin{thm}\label{thm_girth_5}
Let $G$ be a connected simple graph with $\operatorname{girth}(G) \ge 5$. Then the following conditions are equivalent:
\begin{enumerate}
    \item $G^d$ is well-covered for all $d \ge 1$;
    \item $G$ is isomorphic to one of $K_1$, $K_2$, $C_5$, or $C_7$.
\end{enumerate}
\end{thm}

Furthermore, we show that this property is preserved under clique blow-ups (clique substitutions). Recall the definition of graph composition:

\begin{defn}
Let $G$ be a simple graph with vertex set $V(G) = \{1, \ldots, n\}$, and let $H_1, \ldots, H_n$ be graphs. The \emph{composition graph} $G[H_1, \ldots, H_n]$ is obtained by replacing each vertex $i \in V(G)$ with $H_i$ and joining every vertex of $H_i$ to every vertex of $H_j$ whenever $\{i, j\} \in E(G)$.
\end{defn}

We prove that replacing vertices with complete graphs preserves the property that all powers are well-covered:

\begin{thm}\label{thm_clique_blowup}
Let $G$ be a simple graph with $V(G) = \{v_1, \ldots, v_n\}$. For any positive integers $a_1, \ldots, a_n$, let $H = G[K_{a_1}, \ldots, K_{a_n}]$ be the clique blow-up of $G$. Then $G^d$ is well-covered for all $d \ge 1$ if and only if $H^d$ is well-covered for all $d \ge 1$.
\end{thm}

The paper is organized as follows. In Section~\ref{sec_prelim}, we collect necessary notation, definitions, and preliminary results on well-covered graphs and graph powers, and establish several basic properties of graphs all of whose powers are well-covered. Section~\ref{sec_mains} is devoted to proving our main classification results for chordal, very well-covered, and cochordal graphs, as well as graphs of girth at least $5$, and to establishing stability under clique blow-ups.

\section{Graphs, Powers, and Well-coveredness}\label{sec_prelim}

\subsection{Graphs}

\begin{defn}
Let $G = (V(G), E(G))$ be a simple graph on the vertex set $V(G) = \{1, \ldots, n\}$.
\begin{enumerate}
    \item The \emph{neighborhood} of a vertex $x \in V(G)$ is $N_G(x) = \{y \in V(G) \mid \{x, y\} \in E(G)\}$. The \emph{closed neighborhood} of $x$ is $N_G[x] = N_G(x) \cup \{x\}$.
    \item For a subset $U \subseteq V(G)$, the \emph{neighborhood} and \emph{closed neighborhood} of $U$ in $G$ are defined respectively by 
    \[
    N_G(U) = \bigcup_{x \in U} N_G(x) \quad \text{and} \quad N_G[U] = \bigcup_{x \in U} N_G[x].
    \]
    \item A graph $H$ is a \emph{subgraph} of $G$ if $V(H) \subseteq V(G)$ and $E(H) \subseteq E(G)$. It is an \emph{induced subgraph} of $G$ if $E(H) = \{\{u, v\} \in E(G) \mid u, v \in V(H)\}$. For a subset $U \subseteq V(G)$, $G[U]$ denotes the induced subgraph of $G$ on $U$, and $G \setminus U$ denotes $G[V(G) \setminus U]$.
    \item A \emph{path} $P_k$ on $k$ vertices has vertex set $V(P_k) = \{v_1, \ldots, v_k\}$ and edge set $E(P_k) = \{\{v_1, v_2\}, \{v_2, v_3\}, \ldots, \{v_{k-1}, v_k\}\}$. A \emph{cycle} $C_k$ on $k$ vertices ($k \ge 3$) has edge set $E(C_k) = E(P_k) \cup \{\{v_1, v_k\}\}$.
    \item A graph $G$ is \emph{chordal} if it contains no induced cycle of length four or greater. A graph $G$ is \emph{cochordal} if its complement graph $\overline{G}$ is chordal.
    \item The \emph{girth} of a graph $G$, denoted $\girth(G)$, is the length of a smallest cycle of $G$. If $G$ has no cycle, then $\girth(G) = \infty$.
\end{enumerate}
\end{defn}

\begin{defn}[Distance and Graph Powers]
Let $G$ be a connected graph. The \emph{distance} $d_G(u, v)$ between two vertices $u, v \in V(G)$ is the length of a shortest path in $G$ connecting them. The \emph{diameter} of $G$ is $\operatorname{diam}(G) = \max\{d_G(u, v) \mid u, v \in V(G)\}$.

For a positive integer $d \ge 1$, the \emph{$d$-th power} of $G$, denoted by $G^d$, is the graph defined on the vertex set $V(G^d) = V(G)$ with edge set
\[
E(G^d) = \{\{u, v\} \subseteq V(G) \mid u \neq v \text{ and } d_G(u, v) \le d\}.
\]
\end{defn}

\begin{defn}[Independent Sets, Vertex Covers, and Well-Coveredness]
Let $G$ be a simple graph.
\begin{enumerate}
    \item A subset $I \subseteq V(G)$ is an \emph{independent set} of $G$ if no two vertices in $I$ are adjacent in $G$. An independent set $I$ is \emph{maximal} if it is not properly contained in any other independent set of $G$. The \emph{independence number} of $G$, denoted by $\alpha(G)$, is the cardinality of a maximum independent set of $G$.
    \item A subset $C \subseteq V(G)$ is a \emph{vertex cover} of $G$ if $C \cap e \neq \emptyset$ for every $e \in E(G)$. A vertex cover $C$ is \emph{minimal} if no proper subset of $C$ is a vertex cover of $G$. Note that $I \subseteq V(G)$ is a maximal independent set of $G$ if and only if $V(G) \setminus I$ is a minimal vertex cover of $G$.
    \item A graph $G$ is \emph{well-covered} if every minimal vertex cover of $G$ has the same cardinality.
\end{enumerate}
\end{defn}

\begin{defn} 
Let $G$ be a simple graph and $v$ a vertex of $G$. The \emph{eccentricity} of $v$ is defined by $\ecc_G(v) = \max \{d_G(u,v) \mid u \in V(G)\}$. The \emph{radius} of $G$ is defined by $\rad(G) = \min \{ \ecc_G(v) \mid v \in V(G)\}$. A graph $G$ is called \emph{self-centered} if $\rad(G) = \diam(G)$.
\end{defn}

\begin{lem}\label{lem_ecc} 
Let $G$ be a simple connected graph of diameter $D$. Assume that $G^{D-1}$ is well-covered. Then every vertex $v$ of $G$ has eccentricity $D$. In other words, $G$ is self-centered.
\end{lem}

\begin{proof} 
Assume for contradiction that $\ecc_G(v) < D$. Then $N_{G^{D-1}}[v] = V(G)$. Hence, the maximal independent set of $G^{D-1}$ containing $v$ is $\{v\}$, which has size $1$. However, since $\diam(G) = D$, there exist vertices $u, w \in V(G)$ such that $d_G(u,w) = D > D-1$, meaning $u$ and $w$ are independent in $G^{D-1}$. Thus, any maximal independent set containing $\{u, w\}$ has size at least $2$, contradicting the assumption that $G^{D-1}$ is well-covered. The conclusion follows.
\end{proof}

\begin{cor}\label{cor_no_leaves}
Let $G$ be a simple connected graph of diameter $D$. Assume that $D \ge 2$ and $G^{D-1}$ is well-covered. Then $G$ has no leaves.
\end{cor}

\begin{proof}
Suppose for the sake of contradiction that $G$ contains a leaf $v$ with unique neighbor $u$. Then for every vertex $w \notin \{u, v\}$, we have $d_G(v, w) = 1 + d_G(u, w)$. Consequently, $\operatorname{ecc}_G(v) = \operatorname{ecc}_G(u) + 1$, which contradicts Lemma \ref{lem_ecc}. The conclusion follows.
\end{proof}

Let $G$ be a simple graph and $v \in V(G)$ a vertex. For a non-negative integer $a \ge 0$, the \emph{closed ball} of radius $a$ centered at $v$ is defined as 
\[
B_a(v) = \{u \in V(G) \mid d_G(u,v) \le a\}.
\]
Note that $B_d(v) = N_{G^d}[v]$ is precisely the closed neighborhood of $v$ in $G^d$.

\begin{lem}\label{lem_d2}
Let $G$ be a simple graph such that $G^d$ is well-covered for some $d \ge 1$. Then there exist no vertices $s, t, w \in V(G)$ such that $d_G(s,t) \ge d+1$ and $B_d(s) \cup B_d(t) \subseteq B_d(w)$.
\end{lem}

\begin{proof}
Suppose for the sake of contradiction that there exist vertices $s, t, w \in V(G)$ satisfying $d_G(s,t) \ge d+1$ and $B_d(s) \cup B_d(t) \subseteq B_d(w)$. Let $X$ be a maximal independent set of $G^d$ containing $w$. Since $G^d$ is well-covered, we have $|X| = \alpha(G^d)$.

By definition of independent sets in graph powers, no vertex in $X \setminus \{w\}$ can belong to $N_{G^d}[w] = B_d(w)$. Since $B_d(s) \cup B_d(t) \subseteq B_d(w)$, it follows that 
\[
(X \setminus \{w\}) \cap (B_d(s) \cup B_d(t)) = \emptyset.
\]
Furthermore, because $d_G(s,t) \ge d+1$, $\{s, t\}$ is an independent set in $G^d$. Consequently, the set $Y = (X \setminus \{w\}) \cup \{s, t\}$ is an independent set of $G^d$. However, 
\[
|Y| = |X| - 1 + 2 = |X| + 1 = \alpha(G^d) + 1,
\]
which contradicts the assumption that $\alpha(G^d)$ is the maximum size of an independent set in $G^d$. Thus, no such vertices $s, t, w$ can exist.
\end{proof}

\section{Well-coveredness of graph powers}\label{sec_mains}

\subsection{Chordal Graphs}
We begin by recalling several fundamental concepts in chordal graph theory.

\begin{defn}
Let $G$ be a graph. A vertex $v \in V(G)$ is called a \emph{simplicial vertex} if its closed neighborhood $N_G[v]$ induces a clique in $G$. A clique $K \subseteq V(G)$ is called a \emph{simplex} of $G$ if $K = N_G[v]$ for some simplicial vertex $v \in V(G)$.
\end{defn}
We have the following fundamental theorem due to Prisner, Topp, and Vestergaard \cite{PTV}:

\begin{thm}[Well-covered chordal graphs]\label{thm_well_covered_chordal}
A chordal graph $G$ is well-covered if and only if every vertex of $G$ lies in exactly one simplex.
\end{thm}

\begin{proof}[Proof of Theorem \ref{thm_chordal}] 
The implications $(1) \implies (2)$ and $(3) \implies (1)$ are obvious. It remains to prove that if $G$ and $G^2$ are well-covered, then $G$ is a complete graph. 

By Theorem \ref{thm_well_covered_chordal}, since $G$ is a well-covered chordal graph, its vertices are partitioned into pairwise disjoint simplices $S_1, \ldots, S_m$. For each $1 \le i \le m$, let $s_i \in S_i$ be a simplicial vertex. Then $\{s_1, \ldots, s_m\}$ is an independent set in $G^2$ because $N_G[s_i] = S_i$ and $S_i \cap S_j = \emptyset$ for all $i \neq j$, meaning $d_G(s_i, s_j) \ge 3$. Moreover, every vertex in $G$ belongs to some $S_i = N_G[s_i]$ and is thus within distance $1$ or $2$ from $s_i$. Therefore, $\{s_1, \ldots, s_m\}$ is a maximal independent set of $G^2$, so $\alpha(G^2) = m$.

Suppose for the sake of contradiction that $G$ is not a complete graph, so $m \ge 2$ and there exists a non-simplicial vertex $v \in V(G)$. Without loss of generality, assume $v \in S_a$. Since $v$ is non-simplicial, $N_G[v]$ is not a clique, so $v$ must have a neighbor $w$ outside $S_a$; say $w \in S_b$ for some $b \neq a$. Since $S_a = N_G[s_a]$ and $S_b = N_G[s_b]$ are cliques, every vertex in $S_a \cup S_b$ is within distance at most $2$ from $v$ in $G$, so $S_a \cup S_b \subseteq N_{G^2}[v]$. 

Consequently, any maximal independent set $X$ of $G^2$ containing $v$ cannot contain any vertices from $S_a \cup S_b$. Since $X$ can pick at most one vertex from each remaining simplex $S_k$ ($k \notin \{a, b\}$), we obtain $|X| \le m - 1 < \alpha(G^2)$, contradicting the assumption that $G^2$ is well-covered. Thus, $G$ contains no non-simplicial vertices, so $m = 1$ and $G$ is a complete graph.
\end{proof}

\subsection{Very well-covered graphs} We recall the following characterization of very well-covered graphs due to Favaron \cite[Theorem 1.2]{F}.

\begin{thm}[Very well-covered graphs]\label{thm_very_well_char} 
A graph $G$ with a perfect matching $M$ is very well-covered if and only if every edge $\{u, u'\} \in M$ satisfies the following conditions:
\begin{enumerate}
    \item $N_G(u) \cap N_G(u') = \emptyset$;
    \item every vertex $x \in N_G(u) \setminus \{u'\}$ is adjacent to every vertex $y \in N_G(u') \setminus \{u\}$.
\end{enumerate}
\end{thm}

\begin{proof}[Proof of Theorem \ref{thm_very_well}]
The implications $(1) \implies (2)$ and $(3) \implies (1)$ are obvious. It remains to prove that if $G$ is very well-covered and $G^2$ is well-covered, then $G$ is a complete bipartite graph.

Let $M$ be a perfect matching of $G$ satisfying the conditions of Theorem~\ref{thm_very_well_char}. For each vertex $x \in V(G)$, denote by $x'$ its unique matched partner under $M$. 

\smallskip
\noindent\textbf{Step 1:} \emph{Construction of neighborhood sets $S$ and $T$.}

Let $u \in V(G)$ be a vertex of minimum degree in $G$. Set $S = N_G(u)$ and $T = N_G(u')$. 

For any vertex $x \in S = N_G(u)$, applying condition (2) of Theorem~\ref{thm_very_well_char} to the matching edge $\{x, x'\}$ with $u \in N_G(x)$ yields $N_G(x') \setminus \{x\} \subseteq N_G(u)$. Since $x \in N_G(u)$, it follows that $N_G(x') \subseteq N_G(u)$. By the minimality of $\deg_G(u)$, we obtain $N_G(x') = N_G(u) = S$.

Next, applying condition (2) to the matching edge $\{u, u'\}$ with $x \in N_G(u)$ implies $N_G(u') \setminus \{u\} \subseteq N_G(x)$. Since $x \in N_G(u)$, we also have $u \in N_G(x)$, whence $N_G(u') \subseteq N_G(x)$. Conversely, applying condition (2) to the matching edge $\{x, x'\}$ with $u' \in N_G(x')$ gives $N_G(x) \subseteq N_G(u')$. Thus, 
\[
N_G(x) = N_G(u') = T \quad \text{for all } x \in S.
\]
In particular, $S$ and $T$ are disjoint independent sets, and every vertex in $S$ is adjacent to every vertex in $T$. Furthermore, for any $w \in T$, since $u' \in N_G(w)$, applying condition (2) to the matching edge $\{w, w'\}$ yields $N_G(w') \subseteq N_G(u') = T$.

\smallskip
\noindent\textbf{Step 2:} \emph{If there exists $w \in T$ such that $w' \notin S$, then $G^2$ is not well-covered.}

Suppose there exists $w \in T$ such that $w' \notin S$. Since $w' \notin S = N_G(u)$ and $N_G(w') \subseteq T = N_G(u')$, the sequence $(u, u', w, w')$ forms an induced path of length $3$ in $G$, so $d_G(u, w') = 3$. We compare the closed $2$-balls centered at $u$, $u'$, and $w'$:
\begin{align*}
    B_2(u) &= \{u\} \cup S \cup T, \\
    B_2(u') &= \{u'\} \cup T \cup N_G(T), \\
    B_2(w') &= \{w'\} \cup N_G(w') \cup N_G(N_G(w')).
\end{align*}
Since $N_G(w') \subseteq T$, we have $N_G(N_G(w')) \subseteq N_G(T)$, which implies $B_2(w') \subseteq B_2(u')$. Moreover, $B_2(u) \subseteq B_2(u')$ clearly holds. Therefore,
\[
B_2(u) \cup B_2(w') \subseteq B_2(u') \quad \text{with } d_G(u, w') \ge 3,
\]
which contradicts Lemma~\ref{lem_d2}. Hence, $G^2$ cannot be well-covered under this assumption.

\smallskip
\noindent\textbf{Step 3:} \emph{If $w' \in S$ for all $w \in T$, then $G \cong K_{n,n}$.}

Suppose $w' \in S$ for all $w \in T$. Then for every $w \in T$, $w'$ is a vertex in $S$, so $N_G(w') = T$. 

Now assume there exists some $w \in T$ having a neighbor $z \notin S$. Since $w' \in S$, the pair $\{w, w'\}$ is a matching edge with $N_G(w') = T$. By condition (2) applied to $\{w, w'\}$, the neighbor $z \in N_G(w)$ must be adjacent to every vertex in $N_G(w') = T$. In particular, $z$ is adjacent to $u$, which implies $z \in N_G(u) = S$, a contradiction. 

Therefore, $N_G(w) = S$ for all $w \in T$, and $N_G(s) = T$ for all $s \in S$. Consequently, $S \cup T$ forms a connected component of $G$. Since $G$ is connected, $V(G) = S \cup T$. Finally, since $G$ is very well-covered, $|S| = |T| = n$, and thus $G \cong K_{n,n}$.

The conclusion follows.
\end{proof}

\subsection{Cochordal graphs}
The case of cochordal graphs will follow from the following lemma.

\begin{lem}\label{lem_rad_cochor} 
Let $G$ be a connected cochordal graph. Then $\operatorname{rad}(G) \le 2$.
\end{lem}

\begin{proof} 
Let $H = \overline{G}$ be the complement graph of $G$. By definition, $H$ is a chordal graph. Let $s$ be a simplicial vertex of $H$, and let $Q = N_H[s]$ be the corresponding clique in $H$. By the definition of graph complementation, all vertices $u \notin Q$ are adjacent to $s$ in $G$.

Now, let $t \in Q \setminus \{s\}$. Suppose for the sake of contradiction that $N_G(t) \subseteq Q$. This implies that $N_H[t] \supseteq V(G) \setminus Q$. Since $Q$ is a clique in $H$, $t$ is also adjacent in $H$ to every vertex in $Q \setminus \{t\}$, so $N_H(t) \supseteq Q \setminus \{t\}$. Consequently, $N_H(t) = V(G) \setminus \{t\}$, which means $t$ has no neighbors in $G$. Thus, $t$ is an isolated vertex in $G$, contradicting the assumption that $G$ is connected.

Therefore, every $t \in Q \setminus \{s\}$ must have a neighbor $z \notin Q$ in $G$. Since $s$ is adjacent to $z$ in $G$, the sequence $(s, z, t)$ forms a path of length $2$ in $G$, so $d_G(s, t) \le 2$. Because $s$ is also adjacent in $G$ to every vertex outside $Q$, we obtain $\operatorname{ecc}_G(s) \le 2$. It follows that $\operatorname{rad}(G) \le \operatorname{ecc}_G(s) \le 2$.
\end{proof}

\begin{proof}[Proof of Theorem \ref{thm_cochor}] 
The implications $(1) \implies (2)$ and $(3) \implies (1)$ are obvious. The implication $(2) \implies (3)$ follows directly from Lemma \ref{lem_rad_cochor} and Lemma \ref{lem_ecc}.
\end{proof}

\subsection{Graphs of girth $\ge 5$}
We first recall the characterization of well-covered graphs of girth at least $5$ due to Finbow, Hartnell, and Nowakowski~\cite{FHN}. A $5$-cycle of a graph $G$ is called \emph{basic} if it does not contain two adjacent vertices of degree at least $3$ in $G$. A graph $G$ belongs to the class $\mathcal{PC}$ if $V(G)$ can be partitioned into two subsets $\mathcal{P}$ and $\mathcal{C}$ such that $\mathcal{P}$ consists of the vertices incident with pendant edges that form a perfect matching of $\mathcal{P}$, and $\mathcal{C}$ is partitioned by pairwise vertex-disjoint basic $5$-cycles. We now prove the classification of graphs of girth at least $5$ all of whose powers are well-covered.

\begin{proof}[Proof of Theorem \ref{thm_girth_5}]
That $(2)$ implies $(1)$ can be checked directly. Conversely, assume that $G^d$ is well-covered for all $d \ge 1$. If $|V(G)| \le 2$, then $G$ is isomorphic to $K_1$ or $K_2$, so assume that $|V(G)| \ge 3$. By the characterization of Finbow, Hartnell, and Nowakowski~\cite{FHN}, $G$ belongs to $\mathcal{PC}$ or is isomorphic to one of $P_{10}$, $P_{13}$, $Q_{13}$, or $P_{14}$. One can verify that $P_{10}^2$, $P_{13}^2$, $Q_{13}^2$, and $P_{14}^2$ are not well-covered. Hence $G \in \mathcal{PC}$.

By Corollary~\ref{cor_no_leaves}, $G$ has no leaves. Hence, $V(G)$ is partitioned into $m \ge 1$ pairwise vertex-disjoint basic $5$-cycles $C^1, \ldots, C^m$. Suppose for contradiction that $m \ge 2$. We call a vertex of $C^i$ an \emph{attachment vertex} if its degree is at least $3$. By the definition of a basic $5$-cycle, each $C^i$ has at most two nonadjacent attachment vertices. Hence, we may label the vertices of each $C^i$ as $v^i_1, \ldots, v^i_5$ so that $v^i_2$, $v^i_4$, and $v^i_5$ have degree $2$.

\smallskip
\noindent\textbf{Step 1.} $\alpha(G^2) = m$. Every $C^i$ has diameter $2$ and induces a clique in $G^2$, so $\alpha(G^2) \le m$. Conversely, $d_G(v_2^i, v_2^j) \ge 3$ for all $i \neq j$, because any path between $v_2^i$ and $v_2^j$ must pass through at least one attachment vertex in $C^i$ and one in $C^j$. Hence, $\alpha(G^2) \ge m$.

\smallskip
\noindent\textbf{Step 2.} For every $i \in \{1,\ldots,m\}$, every $x \in N_G(v_1^i) \setminus C^i$, and every $y \in N_G(v_3^i) \setminus C^i$, we have $d_G(x,y) \le 2$. Indeed, since $xv_1^i, yv_3^i \in E(G)$, we have $C^i \subseteq B_2(x) \cup B_2(y)$. Thus, if $d_G(x,y) \ge 3$, then $\{x,y\}$ is an independent set in $G^2$, and any independent set of $G^2$ containing $\{x,y\}$ must avoid $C^i$, hence has cardinality at most $m-1$, contradicting the fact that $G^2$ is well-covered.

\smallskip
\noindent\textbf{Step 3.} Contradiction via the well-coveredness of $G^3$. Since $G$ is connected and $m \ge 2$, there exists an edge joining two distinct basic $5$-cycles. Relabeling if necessary, we may assume that $v_1^1 v_1^2 \in E(G)$. Consider the three vertices $s = v_4^1$, $t = v_5^2$, and $w = v_1^2$. We will prove that $d_G(s,t) \ge 4$ and $B_3(s) \cup B_3(t) \subseteq B_3(w)$, contradicting Lemma~\ref{lem_d2}.

\begin{enumerate}
    \item $d_G(s,t) \ge 4$. Any $(v_4^1, v_5^2)$-path must leave $C^1$ via $v_1^1$ or $v_3^1$ and enter $C^2$ via $v_1^2$ or $v_3^2$. Thus, the only way to achieve length at most $3$ is via an edge $v_3^1 v_1^2 \in E(G)$. That would form a $4$-cycle $(v_1^1, v_2^1, v_3^1, v_1^2)$, contradicting $\girth(G) \ge 5$.
    \item $B_3(t) \subseteq B_3(w)$. Every vertex $z \in C^2$ satisfies $d_G(v_1^2, z) \le 2$. For any vertex $z \notin C^2$, every path from $v_5^2$ to $z$ leaves $C^2$ via $v_1^2$ or $v_3^2$, so
    $$d_G(v_5^2, z) \ge \min(1 + d_G(v_1^2, z), 2 + d_G(v_3^2, z)) \ge d_G(v_1^2, z)$$
    because $d_G(v_1^2, z) \le d_G(v_1^2, v_3^2) + d_G(v_3^2, z) = 2 + d_G(v_3^2, z)$.
    \item $B_3(s) \subseteq B_3(w)$. Let $z \in B_3(v_4^1)$. If $z \in C^1$, then $d_G(v_1^2, z) \le d_G(v_1^2, v_1^1) + d_G(v_1^1, z) \le 1 + 2 = 3$. If $z \notin C^1$, then a shortest path of length at most $3$ from $v_4^1$ to $z$ leaves $C^1$ through either $v_1^1$ or $v_3^1$. If it leaves through $v_1^1$, then $2 + d_G(v_1^1, z) \le 3$, so $d_G(v_1^1, z) \le 1$, and hence $d_G(v_1^2, z) \le 2$. If it leaves through $v_3^1$, then $d_G(v_3^1, z) \le 2$. Let such a shortest path from $v_3^1$ to $z$ be $v_3^1, y, z$, where $y$ may coincide with $z$. Since $z \notin C^1$ and $v_4^1, v_2^1$ have degree $2$, we must have $y \notin C^1$. Applying Step~2 to $C^1$ with $v_1^2 \in N_G(v_1^1) \setminus C^1$ and $y \in N_G(v_3^1) \setminus C^1$ yields $d_G(v_1^2, y) \le 2$. Therefore,
    $$d_G(v_1^2, z) \le d_G(v_1^2, y) + d_G(y, z) \le 3.$$
\end{enumerate}
This concludes the proof of the theorem.
\end{proof}

\subsection{Clique substitutions}
\begin{proof}[Proof of Theorem \ref{thm_clique_blowup}]
Let $V(G) = \{v_1, \ldots, v_n\}$. By definition of the clique blow-up, $H = G[K_{a_1}, \ldots, K_{a_n}]$ is obtained by replacing each vertex $v_i \in V(G)$ with a clique $Q_i \cong K_{a_i}$, where $a_i \ge 1$. Two distinct vertices $x \in Q_i$ and $y \in Q_j$ are adjacent in $H$ if and only if $i = j$ or $v_i v_j \in E(G)$.

Define the canonical projection map $\pi: V(H) \to V(G)$ by setting $\pi(x) = v_i$ for all $x \in Q_i$.

\smallskip
\noindent\textbf{Claim 1:} \emph{For any $x, y \in V(H)$ with $\pi(x) \neq \pi(y)$, we have $d_H(x, y) = d_G(\pi(x), \pi(y))$.}

\begin{proof}[Proof of Claim 1]
Let $x \in Q_i$ and $y \in Q_j$ with $i \neq j$. If $v_i = u_0, u_1, \dots, u_k = v_j$ is a path of length $k$ in $G$, we can construct a path of length $k$ in $H$ from $x$ to $y$ by selecting an arbitrary representative $x_r \in Q_{r}$ for each $1 \le r \le k-1$ and setting $x_0 = x$ and $x_k = y$. Thus, $d_H(x, y) \le d_G(v_i, v_j)$. Conversely, if $x = y_0, y_1, \dots, y_m = y$ is a path of length $m$ in $H$, then $\pi(y_0), \pi(y_1), \dots, \pi(y_m)$ yields a walk (and hence contains a path) from $v_i$ to $v_j$ in $G$ of length at most $m$. Therefore, $d_G(v_i, v_j) \le d_H(x, y)$, establishing $d_H(x, y) = d_G(v_i, v_j)$.
\end{proof}

Fix an integer $d \ge 1$. By Claim 1, for any two distinct vertices $x \in Q_i$ and $y \in Q_j$, we have $d_H(x, y) \le d$ if and only if $i = j$ or $d_G(v_i, v_j) \le d$. In other words, two distinct vertices $x, y \in V(H)$ are adjacent in $H^d$ if and only if either $x, y$ belong to the same clique $Q_i$, or their projections $\pi(x), \pi(y)$ are adjacent in $G^d$.

\smallskip
\noindent\textbf{Claim 2:} \emph{A subset $I \subseteq V(H)$ is a maximal independent set of $H^d$ if and only if $I$ contains exactly one vertex from $Q_i$ for each $v_i \in J$, where $J = \pi(I)$ is a maximal independent set of $G^d$.}

\begin{proof}[Proof of Claim 2]
Suppose $I$ is a maximal independent set of $H^d$. 
\begin{enumerate}
    \item Since every clique $Q_i$ induces a clique in $H^d$, $I$ can contain at most one vertex from each $Q_i$. Thus, the restriction of $\pi$ to $I$ is injective, so $|\pi(I)| = |I|$.
    \item For any two distinct vertices $x, y \in I$, we have $d_H(x, y) \ge d+1$. By Claim 1, $d_G(\pi(x), \pi(y)) = d_H(x,y) \ge d+1$, which implies that $J = \pi(I)$ is an independent set of $G^d$.
    \item To show $J$ is maximal in $G^d$, suppose there exists $v_k \in V(G) \setminus J$ such that $J \cup \{v_k\}$ is independent in $G^d$. Pick an arbitrary vertex $z \in Q_k$. Then $d_G(v_k, \pi(x)) \ge d+1$ for all $x \in I$, so $d_H(z, x) \ge d+1$ for all $x \in I$. Hence $I \cup \{z\}$ is an independent set of $H^d$, contradicting the maximality of $I$. Thus, $J$ is a maximal independent set of $G^d$.
\end{enumerate}

Conversely, let $J$ be a maximal independent set of $G^d$. For each $v_i \in J$, pick a vertex $x_i \in Q_i$, and set $I = \{x_i \mid v_i \in J\}$. By Claim 1, $d_H(x_i, x_j) = d_G(v_i, v_j) \ge d+1$ for all $i \neq j$, so $I$ is independent in $H^d$. If $I$ were not maximal in $H^d$, there would exist $y \in V(H) \setminus I$ such that $I \cup \{y\}$ is independent in $H^d$. If $y \in Q_i$ for some $v_i \in J$, then $y$ is adjacent to $x_i \in I$ in $H \subseteq H^d$, a contradiction. If $y \in Q_k$ for some $v_k \notin J$, then $d_G(v_k, v_i) = d_H(y, x_i) \ge d+1$ for all $v_i \in J$, which implies $J \cup \{v_k\}$ is an independent set of $G^d$, contradicting the maximality of $J$. Therefore, $I$ is a maximal independent set of $H^d$ with $|I| = |J|$.
\end{proof}

By Claim 2, the cardinalities of all maximal independent sets of $H^d$ are in one-to-one correspondence with the cardinalities of the maximal independent sets of $G^d$. Consequently, all maximal independent sets of $H^d$ have the same cardinality if and only if all maximal independent sets of $G^d$ have the same cardinality. Hence, $G^d$ is well-covered for all $d \ge 1$ if and only if $H^d$ is well-covered for all $d \ge 1$.
\end{proof}

\begin{cor}\label{cor_twins}
Let $G$ be a simple graph. Suppose $u, v \in V(G)$ are true twins in $G$, that is, $\{u,v\} \in E(G)$ and $N_G[u] = N_G[v]$. Let $F$ be the graph obtained by contracting the edge $\{u,v\}$ to a single vertex. Then $G^d$ is well-covered for all $d \ge 1$ if and only if $F^d$ is well-covered for all $d \ge 1$.
\end{cor}

\begin{proof}
Since $u$ and $v$ are true twins, contracting the edge $uv$ to a single vertex $w$ yields a graph $F$ isomorphic to the deletion $G - v$. Consequently, $G$ is isomorphic to the clique blow-up of $F$ obtained by replacing the vertex $w$ with a clique $K_2$ and replacing every other vertex $x \in V(F) \setminus \{w\}$ with $K_1$. The result then follows directly from Theorem \ref{thm_clique_blowup}.
\end{proof}

\begin{prop}\label{prop_ind_blow}
Let $G$ be a connected simple graph with $V(G) = \{v_1, \ldots, v_n\}$, and let $r \ge 1$ be an integer. Let $H = G[r K_1, \ldots, r K_1]$ be the $r$-fold independence blow-up of $G$. 
Then $\operatorname{diam}(H) = \max\{\operatorname{diam}(G), 2\}$ for $r \ge 2$, and $H^d$ is well-covered for all $d \ge 1$ if and only if $G^d$ is well-covered for all $d \ge 1$.
\end{prop}

\begin{proof}
Let $Q_1, \ldots, Q_n$ denote the independent sets of size $r$ in $H$ corresponding to the vertices $v_1, \ldots, v_n$ of $G$, respectively, and let $\pi: V(H) \to V(G)$ be the canonical projection mapping each $x \in Q_i$ to $v_i$. Two distinct vertices $x, y \in V(H)$ are adjacent in $H$ if and only if $\pi(x) \neq \pi(y)$ and $\{\pi(x),\pi(y)\} \in E(G)$.

First, we analyze the distance function $d_H$ on $V(H)$:
\begin{enumerate}
    \item If $x, y \in Q_i$ with $x \neq y$ (which can occur when $r \ge 2$), then $x$ and $y$ are non-adjacent false twins in $H$. Since $G$ is connected, $v_i$ has at least one neighbor $v_j \in V(G)$. Any $z \in Q_j$ is adjacent to both $x$ and $y$ in $H$, so $(x, z, y)$ is a path of length $2$ in $H$. Thus $d_H(x, y) = 2$.
    \item If $x \in Q_i$ and $y \in Q_j$ with $i \neq j$, any shortest path between $v_i$ and $v_j$ in $G$ lifts to a path of the same length between $x$ and $y$ in $H$, and any path between $x$ and $y$ in $H$ projects to a walk of equal length between $v_i$ and $v_j$ in $G$. Thus $d_H(x, y) = d_G(v_i, v_j)$.
\end{enumerate}

Taking the maximum over all pairs $x, y \in V(H)$, it follows that $\operatorname{diam}(H) = \max\{\operatorname{diam}(G), 2\}$ when $r \ge 2$ (and $\operatorname{diam}(H) = \operatorname{diam}(G)$ when $r = 1$).

Next, fix $d \ge 1$. From the distance formula above, two distinct vertices $x, y \in V(H)$ are adjacent in $H^d$ if and only if either $d \ge 2$ and $\pi(x) = \pi(y)$, or $\{\pi(x), \pi(y)\} \in E(G^d)$. 

If $d = 1$, each $Q_i$ is an independent set in $H$, and every maximal independent set $I$ of $H$ contains at most one vertex from each $Q_i$. Hence, $\pi$ maps maximal independent sets of $H$ bijectively to maximal independent sets of $G$, preserving cardinality. If $d \ge 2$, each $Q_i$ becomes a clique in $H^d$, so any maximal independent set $I$ of $H^d$ contains exactly one vertex from each $Q_i$ for $v_i \in J$, where $J = \pi(I)$ is a maximal independent set of $G^d$.

In both cases, all maximal independent sets of $H^d$ have equal cardinality if and only if all maximal independent sets of $G^d$ have equal cardinality. Therefore, $H^d$ is well-covered for all $d \ge 1$ if and only if $G^d$ is well-covered for all $d \ge 1$.
\end{proof}

\begin{exam}\label{exam_c5_blowups}
Consider the following examples of clique and independence blow-ups of $C_5$:
\begin{enumerate}[(1)]
    \item Let $H_1 = C_5[K_2, K_1, K_1, K_1, K_1]$ be the clique blow-up of $C_5$ obtained by replacing a single vertex with a clique $K_2$.
    \item Let $H_2 = C_5[2K_1, 2K_1, 2K_1, 2K_1, 2K_1]$ be the $2$-fold independence blow-up of $C_5$ obtained by replacing every vertex with an independent set $2K_1$.
\end{enumerate}
By Theorem \ref{thm_clique_blowup} and Proposition \ref{prop_ind_blow}, both $H_1^d$ and $H_2^d$ are well-covered for all $d \ge 1$, as $C_5^d$ is well-covered for all $d \ge 1$.
\end{exam}

\begin{figure}[htbp]
\centering
\begin{tikzpicture}[scale=1.1, every node/.style={circle, draw, fill=black, inner sep=1.8pt}]

    \begin{scope}[xshift=-3.5cm]
        \node (v2) at (18:1.5) {};
        \node (v3) at (290:1.5) {};
        \node (v4) at (250:1.5) {};
        \node (v5) at (162:1.5) {};
        
        \node (v11) at (0, 1.8) {};
        \node (v12) at (0, 1.0) {};

        \draw (v2) -- (v3) -- (v4) -- (v5);
        
        \draw[thick] (v11) -- (v12);
        
        \draw (v11) -- (v2);
        \draw (v11) -- (v5);
        \draw (v12) -- (v2);
        \draw (v12) -- (v5);

        \node[draw=none, fill=none] at (0,-2.3) {$H_1$};
    \end{scope}

    \begin{scope}[xshift=3.5cm]
        \foreach \i/\ang in {1/90, 2/18, 3/290, 4/250, 5/162} {
            \node (u\i) at (\ang:1.6) {};
        }
        
        \foreach \i/\ang in {1/90, 2/18, 3/290, 4/250, 5/162} {
            \node (w\i) at (\ang:0.9) {};
        }

        \foreach \i/\j in {1/2, 2/3, 3/4, 4/5, 5/1} {
            \draw (u\i) -- (u\j);
            \draw (u\i) -- (w\j);
            \draw (w\i) -- (u\j);
            \draw (w\i) -- (w\j);
        }

        \node[draw=none, fill=none] at (0,-2.3) {$H_2$};
    \end{scope}

\end{tikzpicture}
\caption{The clique blow-up $H_1$ and the $2$-fold independence blow-up $H_2$ of $C_5$.}
\label{fig_c5_blowups}
\end{figure}
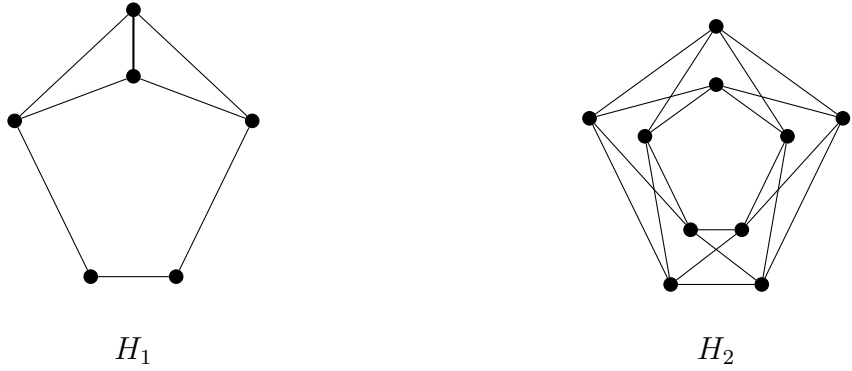

The following example illustrates that there exist graphs of diameter $3$ whose powers are all well-covered.

\begin{exam}
Let $G = C_7$, which has diameter $3$. Then $G^d$ is well-covered for all $d \ge 1$. Similarly, if $G = C_5 \square K_2$ (the pentagonal prism), then $\operatorname{diam}(G) = 3$ and $G^d$ is well-covered for all $d \ge 1$.
\end{exam}

By Corollary \ref{cor_twins}, when investigating graphs $G$ such that $G^d$ is well-covered for all $d \ge 1$, we may assume without loss of generality that $G$ has no true twins. We do not know of any examples of such graphs with diameter at least $4$. On the other hand, every well-covered graph of diameter $2$ automatically satisfies this condition. This leads us to pose the following question.

\begin{quest}\label{quest_diam}
Let $G$ be a simple connected graph such that $G^d$ is well-covered for all $d \ge 1$. Is it true that $\operatorname{diam}(G) \le 3$? Furthermore, if $\operatorname{diam}(G) = 3$, are there only finitely many such graphs without true twins?
\end{quest}

\begin{rem}
    Ma, Vu, and Zhang~\cite{MVZ} constructed graphs with the well-covered powers property of arbitrarily large diameter via the Cartesian product, answering Question~\ref{quest_diam} in the negative.
\end{rem}

\section*{Acknowledgment} 
A part of this work was completed during the first authors' visit to the Vietnam Institute for Advanced Study in Mathematics (VIASM). She thanks VIASM for its support and hospitality.

  \section*{Declarations}
  
  
  \subsection*{\bf Ethical approval}\ \\[-0.3cm]
  
  \noindent Not applicable.
  
  
  \subsection*{\bf Competing interests}\ \\[-0.3cm]
  
  \noindent The authors declare that there are no conflicts of interest.
  
  
  \subsection*{\bf Authors' contributions}\ \\[-0.3cm]
  
  \noindent All authors have contributed equally to this work.
  

  
  \subsection*{\bf Availability of data and materials}\ \\[-0.3cm]
  
  \noindent Data sharing is not applicable to this work, as no data sets were generated or analyzed during the current study.


\begin{thebibliography}{99}

\bibitem[BM]{BM} A. Bondy, U. S. R. Murty, \textit{Graph theory}, Graduate Texts in Mathematics {\bf 244}  (2008), Springer.
 
\bibitem[D]{D}
Z. Deniz,
\emph{On the well-coveredness of square graphs},
Comm. Fac. Sci. Univ. Ank. Ser. A1 Math. Stat. {\bf 71} (2022), 490--501.

\bibitem[F]{F} O. Favaron, \textit{Very well covered graphs}, Discrete Mathematics {\bf 42} (1982), 177--187.

\bibitem[FHN]{FHN} A. Finbow, B. Hartnell, R. J. Nowakowski, \emph{A characterization of well covered graphs of girth 5 or greater}, Journal of Combinatorial Theory, Series B {\bf 57} (1) (1993), 44–68.


\bibitem[MVZ]{MVZ} H. Ma, T. Vu, and W. Zhang, \emph{Well-covered powers: Catersian products and Cayley graphs}, arXiv:2611.03956 (2026).

\bibitem[P]{P} M. D. Plummer, \emph{Well-covered graphs: a survey}, Quaestiones Mathematicae, 1993 16(3), 253-287.
 

 \bibitem[PTV]{PTV} E. Prisner, J. Topp, and P. D. Vestergaard. \emph{Well covered simplicial, chordal, and circular arc graphs}. Journal of Graph Theory {\bf 21(2)} (1996), 113-119.

\bibitem[R]{R} G. Ravindra, \emph{Well-covered graphs}, Journal of Combinatorics and System Sciences, 2 (1) (1977), 20–21.

\end{thebibliography}
\end{document}